\documentclass[11pt]{amsart}

\usepackage{amssymb,latexsym}
\usepackage{amsfonts}
\usepackage{amsthm}
\usepackage{enumerate}
\usepackage{dsfont}

\newtheorem{theorem}{Theorem}[section]
\newtheorem{prop}[theorem]{Proposition}
\newtheorem{lemma}[theorem]{Lemma}
\newtheorem{codinglemma}[theorem]{The Coding Lemma}

\theoremstyle{definition}

\newtheorem{question}[theorem]{Question}

\theoremstyle{remark}
\newtheorem{remark}[theorem]{Remark}

\numberwithin{equation}{section}

\DeclareMathOperator{\lh}{lh} \DeclareMathOperator{\tr}{tr}
\DeclareMathOperator{\dom}{dom}

\newcommand{\cantor}{2^\omega}
\newcommand{\inx}[1]{\ \epsilon_{#1}\ }

\newcommand{\rest}[1]{\upharpoonright{#1}}
\newcommand{\restrict}{\upharpoonright}

\def\R{{\mathbb R}}

\def\Q{{\mathbb Q}}

\begin{document}

\title{A co-analytic maximal set of orthogonal measures}

%    Information for first author
\author{Vera Fischer}
%    Address of record for the research reported here
\address{Kurt G\"odel Research Center, University of Vienna, W\"ahringer Strasse 25, 1090 Vienna, Austria}
\email{vfischer@logic.univie.ac.at}

%    Information for second author
\author{Asger T\"ornquist}
\address{Kurt G\"odel Research Center, University of Vienna, W\"ahringer Strasse 25, 1090 Vienna, Austria}
\email{asger@logic.univie.ac.at}

\thanks{The authors wish to thank the Austrian Science Fund FWF for post-doctoral support through grant no. P 19375-N18 (T\"ornquist) and P 20835-N13 (Fischer).}

%    General info

%\date{}


\begin{abstract}
We prove that if $V=L$ then there is a $\Pi^1_1$ maximal orthogonal
(i.e. mutually singular) set of measures on Cantor space. This
provides a natural counterpoint to the well-known Theorem of Preiss
and Rataj \cite{preissrataj85} that no analytic set of measures can
be maximal orthogonal.
\end{abstract}

\maketitle

\section{Introduction}

Let $X$ be a Polish space and let $P(X)$ be the associated Polish
space of Borel probability measures on $X$ (see e.g.
\cite[17.E]{kechris95}). Recall that $\mu,\nu\in P(X)$ are said to
be {\it orthogonal} (or {\it mutually singular}) if there is a Borel
set $B\subseteq X$ such that $\mu(B)=1$ and $\nu(B)=0$. We will
write $\mu\perp\nu$.

Preiss and Rataj proved in \cite{preissrataj85} that if $X$ is an
uncountable Polish space then no analytic set of measures can be
maximal orthogonal, answering a question raised by Mauldin. Later
Kechris and Sofronidis \cite{kecsof01} gave a new proof of this
result using Hjorth's theory of turbulence.

The purpose of this paper is to prove that Preiss and Rataj's result
is in some sense optimal. Specifically, we will prove:

\begin{theorem}
If $V=L$ then there is a $\Pi^1_1$ maximal set of orthogonal
measures in $P(\cantor)$. \label{mainthmv1}
\end{theorem}

The assumption that $V=L$ can of course be replaced by the
assumption that all reals are constructible. Also, the proof easily
relativizes to a parameter $x\in\cantor$: If $V=L[x]$ then there is
a $\Pi^1_1(x)$ maximal orthogonal set of measures in $P(\cantor)$.

Theorem \ref{mainthmv1} belongs to a line of results starting with
A.W. Miller's paper \cite{miller89}. Miller proved, among several
other results, that assuming $V=L$ there is a $\Pi^1_1$ maximal
almost disjoint family in $\mathcal P(\omega)$, there is a $\Pi^1_1$
Hamel basis for $\R$ over $\Q$, and there is a $\Pi^1_1$ set meeting
every line in $\R^2$ exactly twice.

More recently, Miller's technique has found use in the study of
maximal cofinitary subgroups of the infinite symmetric group
$S_\infty$: Gao and Zhang showed in \cite{gaozha05} that if $V=L$
then there is a maximal cofinitary group generated by a $\Pi^1_1$
subset of $S_\infty$, and Kastermans in \cite{kastermans09} improved
this by showing that if $V=L$ then there is a $\Pi^1_1$ maximal
cofinitary subgroup of $S_\infty$.

\medskip

The present paper is organized into four sections. In \S 2 we
introduce the basic effective descriptive set-theoretic notions
related to the space $P(\cantor)$, in particular, we introduce a
natural notion of a code for a measure on $\cantor$. We also revisit
a product measures construction due to Kechris and Sofronidis.
Theorem \ref{mainthmv1} is proved in \S 3. The proof hinges on a
method for coding a given real into a non-atomic measure while
keeping the {\it measure class} of the original measure intact in
the process; this is the content of the ``Coding Lemma''
\ref{thecodinglemma}. Finally in \S 4 we show that a maximal
orthogonal family of continuous measures always has size continuum,
and that if there is a Cohen real over $L$ in $V$ then there is no
$\Pi^1_1$ maximal set of orthogonal measures.

\medskip

{\it Remark.} In the present paper we have attempted to give a
completely elementary account of Miller's technique as it applies to
Theorem \ref{mainthmv1} above, and to provide the details of the
argument while relying only on standard methods that can be found in
places such as \cite[\S 13]{kanamori97} or \cite[Ch. 5]{drake74}. A
somewhat different exposition of the details of Miller's technique
can be found in Kastermans' thesis \cite{kastermans06}.

\section{Preliminaries}

For $s\in 2^{<\omega}$, let
$$
N_s=\{x\in\cantor: s\subseteq x\},
$$
the basic neighbourhood defined by $s$. Define
$$
p(2^\omega)=\{f:2^{<\omega}\rightarrow[0,1]: f(\emptyset)=1\wedge (\forall
s\in 2^{<\omega})f(s)=f(s^\smallfrown 0)+ f(s^\smallfrown 1)\}.
$$
Then $p(2^\omega)\subseteq [0,1]^{2^{<\omega}}$ is closed, and an
easy application of Kolmogorov's consistency Theorem shows that for
each $f\in p(2^\omega)$ there is a unique $\mu_f\in P(2^\omega)$
such that $\mu_f(N_s)=f(s)$ for all $s\in 2^{<\omega}$, see
\cite[17.17]{kechris95}. Conversely, if $\mu\in P(2^\omega)$ then
$f(s)=\mu(N_s)$ defines $f\in p(2^\omega)$ such that $\mu_f=\mu$,
thus $f\mapsto\mu_f$ is a bijection. We will call the element $f\in
p(\cantor)$  the {\it code} for $\mu_f$.

Note that if $s_n$ enumerates $2^{<\omega}$ and we let
$f_n:2^\omega\rightarrow \mathbb{R}$ be defined as follows:
$$
f_n(x) = \left\{ \begin{array}{ll}
         1 & \mbox{if $s_n\subseteq x$}\\
         0 & \mbox{otherwise},\end{array} \right.
$$
then the metric on $P(2^\omega)$ defined by
$$
\delta(\mu,\nu)=\sum_{n=0}^{\infty} 2^{-n-1}\dfrac{|\int f_n d\mu-\int f_n d\nu|}{\|f_n\|_\infty}
$$
given in \cite[17.19]{kechris95} makes the map $f\mapsto \mu_f$ an
isometric bijection if we equip $p(2^\omega)$ with the metric
$$
d(f,g)=\sum_{n=0}^\infty 2^{-n-1}|f(s_n)-g(s_n)|.
$$
Let $(q_i:i\in\omega)$ be a recursive enumeration of
$$
\{q:2^n\to \mathbb Q\cap[0,1]: n\in\omega\wedge \sum_{s\in\dom(q)}
q(s)=1\}.
$$
For each $q_i$, let $\hat q_i\in p(\cantor)$ be the unique element
of $p(\cantor)$ such that
$$
(\forall s\in \dom(q_i))(\forall k)\hat q_i(s\smallfrown 0^k)=q_i(s),
$$
where $0^k$ denotes a sequence of zeros of length $k$. Clearly the
sequence $(\hat q_i: i\in\omega)$ is dense in $p(\cantor)$, and it
is routine to see that the relations $P,Q\subseteq\omega^4$ defined
by
$$
P(i,j,m,k)\iff d(\hat q_i,\hat q_j)\leq\frac m {k+1}
$$
$$
Q(i,j,m,k)\iff d(\hat q_i,\hat q_j)<\frac m {k+1}
$$
are recursive. Thus $(\hat q_i:i\in\omega)$ provides a recursive
presentation (in the sense of \cite[3B]{moschovakis80}) of
$p(\cantor)$, and so $(\mu_{\hat q_i}: i\in\omega)$ provides a
recursive presentation of $P(\cantor)$. The map $f\mapsto\mu_f$ is
then a recursive isomorphism between $p(\cantor)$ and $P(\cantor)$.
So from a descriptive set-theoretic point of view there is really no
difference between working with $P(\cantor)$ or $p(\cantor)$. In
particular, it doesn't matter in hierarchy complexity calculations
if we deal with the codes for measures, or with the measures
themselves.

\begin{remark}
Although we could easily have given $P(\cantor)$ a recursive
presentation directly without the detour via $p(\cantor)$, the space
$p(\cantor)$ will still be useful to us. Namely, elements of
$P(\cantor)$ are formally functions $\mu:\mathcal B(\cantor)\to
[0,1]$ defined on the Borel sets $\mathcal B(\cantor)$, and so
formally $\mu\notin L_\delta$ for any $\delta<\omega_1$. However,
since codes are simply functions from $2^{<\omega}$ to $[0,1]$, the
code for $\mu$ may be in $L_\delta$ for some $\delta<\omega_1$, even
if $\mu\notin L_\delta$.
\end{remark}

Recall from real analysis that if $\mu,\nu\in P(\cantor)$, then
$\mu$ is {\it absolutely continuous} with respect to $\nu$, written
$\mu \ll \nu$ if for all Borel subset $B$ of $2^\omega$ it holds
that $\nu(B)=0$ implies $\mu(B)=0$. We say that $\mu,\nu\in
P(\cantor)$ are {\it absolutely equivalent}, written
$\mu\approx\nu$, if $\mu\ll\nu$ and $\nu\ll\mu$.

\begin{lemma}
(a) The relations $\ll$, $\approx$ and $\bot$ are arithmetical.

(b) The set
$$
P_c(2^\omega)=\{\mu\in P(2^\omega):\mu\;\hbox{is non-atomic}\}
$$
is arithmetical.
\end{lemma}
\begin{proof}
(a) To see that $\ll$ and $\approx$ are arithmetical, note that by
\cite[p. 105]{kechris95}
\begin{align*}
\mu\ll\nu\iff(\forall\epsilon>0)(\exists \delta>0)(\forall
B\subseteq 2^\omega\;\hbox{Borel}) (\nu(B)<\delta \rightarrow
\mu(B)<\epsilon)
\end{align*}
and so using \cite[17.10]{kechris95} this is equivalent to
\begin{align*}
\mu\ll\nu\iff (\forall\epsilon>0)(\exists\delta>0)(\forall
s_1,\dots, s_n\in 2^{<\omega})&(\nu(\bigcup_{i=1}^n
N_{s_i})<\delta\longrightarrow\\
&\mu(\bigcup_{i=1}^n N_{s_i})<\epsilon).
\end{align*}
To see that $\bot$ is arithmetical, note that
\begin{align*}
\mu\bot\nu\iff &(\forall\epsilon>0)(\exists s_1,\dots, s_n\in
2^{<\omega})(\mu(\bigcup_{i=1}^n N_{s_i})<\epsilon\wedge
\nu(\bigcup_{i=1}^n N_{s_i})>1-\epsilon)
\end{align*}

(b) We claim that
$$
\mu\in P_c(2^\omega)\iff (\forall\epsilon
>0)(\exists n)(\forall s\in 2^{<\omega})(\hbox{lh}(s)=n\rightarrow
\mu(N_s)<\epsilon).
$$
The implication from right to left is clear.
To see the reverse implication, note that the tree
$$
\{s\in
2^{<\omega}:\mu(N_s)>\epsilon\}
$$
is finite branching, so by K\"{o}nig's Lemma it either has finite
height or it has an infinite branch. The latter is the case if and
only if $\mu$ has an atom.
\end{proof}

\begin{remark}
We let
$$
p_c(\cantor)=\{f\in p(\cantor):\mu_f\text{ is non-atomic}\},
$$
which is arithmetical by the above.
\end{remark}

We now recall a construction due to Kechris and Sofronidis \cite[p.
1463f]{kecsof01}, which is based on a result of Kakutani
\cite{kakutani48} regarding the equivalence of product measures. For
$x\in 2^\omega$, define $\alpha^x\in [0,1]^\omega$ by
$$
\alpha^x(n) = \left\{ \begin{array}{ll}
         \frac{1}{4}(1+ \frac{1}{\sqrt{n+1}}) & \mbox{if $x(n)=1$}\\
         \frac{1}{4} & \mbox{if $x(n)=0$}.\end{array} \right.
$$
Then we let $\mu^x\in P(2^\omega)$ be the product measure on
$\cantor$ defined by
$$
\mu^x=\prod_{n=0}^\infty
[\alpha^x(n)\delta_0+(1-\alpha^x(n))\delta_1]
$$
where $\delta_0,\delta_1$ are the point measures on $2=\{0,1\}$. The
function $x\mapsto\mu^x$ is continuous. The corresponding map
$2^\omega \to p(2^\omega):x\mapsto f^x$ such that $\mu^x=\mu_{f^x}$
for all $x$, is given by
$$
f^x(s)=\prod_{k=0}^{\hbox{lh}(s)}[(1-s(k))\alpha^x(k)+s(k)(1-\alpha^x(k))],
$$
and is clearly recursive. For $x,x^\prime\in 2^\omega$, let
$$
xE_I x^\prime\iff
\sum_{n=0}^\infty{\frac{|x(n)-x^\prime(n)|}{n+1}<\infty}.
$$
From Kakutani's theorem we obtain that if $xE_Ix^\prime$ then
$\mu^x\approx \mu^{x^\prime}$ and if $\neg{(x{E_I} x^\prime)}$ then
$\mu^x\bot\mu^{x^\prime}$. (See \cite[p. 1463]{kecsof01}.)  For the
next lemma it is worth recalling that $\mu,\nu\in P(\cantor)$ are
orthogonal if and only if
$$
\neg(\exists\eta\in P(2^\omega)) \eta\ll\mu\wedge\eta\ll\nu.
$$
Moreover, $\ll$ has the {\it ccc below} property: For any $\mu\in
P(\cantor)$, any family of orthogonal measures $\ll$ below $\mu$ is
countable, see e.g. the proof of \cite[Theorem 3.1]{kecsof01}.

\begin{lemma}
(a) If $\mu\in P(2^\omega)$ then
$$
\{x\in\cantor:\mu^x\perp\mu\}
$$
is comeagre.

(b) If $(\mu_n)$ is a finite or countable sequence of measures on
$\cantor$ then there is $\nu\in P(2^\omega)$ such that
$$
(\forall n) \nu\perp\mu_n
$$
and $\nu$ is arithmetical in $(\mu_n)$. \label{ortholemma}
\end{lemma}
\begin{proof}
(a) Since
$$
\{x\in 2^\omega:\neg\mu^x\bot \mu\}
$$
is clearly $E_I$ invariant, if it is non-meagre it must be comeagre.
But then there must be an uncountable sequence $x_\alpha\in
2^\omega$, $\alpha<\omega_1$, such that if $\alpha\neq \beta$ then
$\neg x_\alpha E_I x_\beta$ and $\mu^{x_\alpha}\not\perp \mu$,
contradicting that $\bot$ is ccc below $\mu$.

(b) The set $\{x\in 2^\omega:(\forall n)(\mu^x\bot \mu_n)\}$ is
arithmetical in $(\mu_n)$, and by (a) it is comeagre. Thus the
second claim follows from \cite[4.1.4]{kechris73}.
\end{proof}

\section{Proof of the main theorem}

In this section we prove Theorem \ref{mainthmv1}. It is clearly
enough to establish the following:

\begin{theorem}
If $V=L$ then there is a $\Pi^1_1$ maximal set of orthogonal
measures in $P_c(2^\omega)$.\label{mainthmv2}
\end{theorem}

Then Theorem \ref{mainthmv1} follows by taking a union of a
$\Pi^1_1$ maximal set of orthogonal measures in $P_c(\cantor)$ with
the set of all point measures (Dirac measures), which clearly is a
$\Pi^0_1$ set.

Our notation follows that of \cite[p. 167ff.]{kanamori97}, with very
few differences. For convenience we recall the definitions and facts
that are most important for the present paper.

The canonical wellordering of $L$ will be denoted $<_L$. The
language of set theory (LOST) is denoted $\mathcal L_\epsilon$. If
$x\in\cantor$ then we define a binary relation on $\omega$ by
$$
m\inx{x}n\iff x(\langle m,n\rangle)=1,
$$
where $\langle \cdot,\cdot\rangle$ refers to some standard G\"odel
pairing function of coding a pair of integers by a single integer.
We let
$$
M_x=(\omega,\epsilon_x),
$$
the $\mathcal L_\epsilon$ structure coded by $x$. If $M_x$ is
wellfounded and extensional then we denote by $\tr(M_x)$ the
transitive collapse of $M_x$, and by $\pi_x:M_x\to\tr(M_x)$ the
corresponding isomorphism.

The following proposition encapsulates the basic descriptive
set-theoretic correspondences between $x$, $M_x$ and the
satisfaction relation. We refer to \cite[13.8]{kanamori97} and the
remarks immediately thereafter for a proof.

\begin{prop}
(a) If $\varphi(v_0,\ldots,v_{k-1})$ is a LOST formula with all free
variables shown then
$$
\{(x,n_0\ldots,
n_{k-1})\in\cantor\times\omega\times\cdots\times\omega: M_x\models
\varphi[n_0,\ldots,n_{k-1}]\}.
$$
is arithmetical.

(b) For $x\in\cantor$ such that $M_x$ is wellfounded and
extensional, the relation
$$
\{(m,f)\in\omega\times p(2^\omega): \pi_x(m)=f\}
$$
is arithmetical in $x$. The same holds if we replace $p(2^\omega)$
with $\omega^\omega$, $\cantor$, or other reasonable Polish product
spaces.

(c) There is a LOST sentence $\sigma_0$ such that if
$M_x\models\sigma_0$ and $M_x$ is wellfounded and extensional, then
$M_x\simeq L_\delta$ for some limit ordinal $\delta<\omega_1$.

(d) There is a LOST formula $\varphi_0(v_0,v_1)$ which defines the
canonical wellordering of $L_\delta$ for all
$\delta>\omega$.\label{basicprop}
\end{prop}

{\it Remark.} For $x\in\cantor$ and $n_0,n_1\in\omega$ it will be
convenient to write $n_0<_{\varphi_0}^x n_1$ as an abbreviation of
$M_x\models \varphi_0[n_0,n_1]$. By (a) in the previous proposition
$n_0<_{\varphi_0}^x n_1$ is arithmetical uniformly in $x$.

\medskip

As motivation for the proof of Theorem \ref{mainthmv1}, we first
prove the following easier result:

\begin{prop}
If $V=L$ then there is a $\Delta^1_2$ maximal set of orthogonal
measures.\label{easyprop}
\end{prop}

\begin{proof}
Work in $L$. The construction is done by induction on $\omega_1$. We
choose a sequence $\langle\mu_\beta:\beta<\omega_1\rangle$ such that
at each $\alpha<\omega_1$ it is the case that $\mu_\alpha$ is the
$<_L$-least measure such that $\mu_\alpha\perp\mu_\beta$ for all
$\beta<\alpha$. That $\mu_\alpha$ always exists follows from Lemma
\ref{ortholemma}. Then it is easy to see that
$$
A=\{\mu_\alpha:\alpha<\omega_1\}
$$
is a maximal orthogonal set of measures. To see that $A$ is
$\Delta^1_2$, define a relation $P\subseteq
p(\cantor)^{\leq\omega}\times \cantor$ by letting $P(s,x)$ if and
only if
\begin{enumerate}
\item $M_x$ is wellfounded and transitive, $M_x\models\sigma_0$, and
for some $m\in\omega$ we have $\pi_x(m)=s$.
\item $\{\mu_{s(n)}:n<\lh(s)\}$ is a set of orthogonal measures.
\item For all $n<\lh(s)$ it holds that $s(n)$ is the $<_L$-smallest
code for a measure orthogonal to all $s(k)$ for which $s(k)<_L
s(n)$.
\end{enumerate}
Condition (1) is clearly $\Pi^1_1$, and (2) is arithmetical.
Finally, if (1) and (2) hold then (3) may be expressed by saying
\begin{align*}
&(\forall n<\lh(s))(\forall f\in p(\cantor))(\forall n_1)[(\exists
n_0) n_0<_{\varphi_0}^x n_1\wedge
\pi_x(n_0)=f\wedge\pi_x(n_1)=s(n)]\\
&\longrightarrow [(\exists l)(\exists l')\neg s(l)\perp f\wedge
\pi_x(l')=s(l)\wedge l'<_{\varphi_0}^x n_1].
\end{align*}
Note that $P(s,x)$ holds if and only if $s$ is a sequence of codes
for the measures in some initial segment
$\{\mu_\alpha:\alpha<\beta\}$, and that the inductive construction
of this initial segment is witnessed in $L_\delta\simeq M_x$, for
some limit $\delta<\omega_1$. It then follows that
\begin{align*}
\mu\in A \iff& (\exists s)(\exists x) [P(s,x)\wedge (\exists n)
\mu_{s(n)}=\mu]\\
\iff& (\forall f)(\forall s)(\forall x)[(P(s,x)\wedge
\mu=\mu_f)\longrightarrow \\
&((\forall l) s(l)<_L f\vee (\exists l) s(l)=\mu)].
\end{align*}
Since the reference to $<_L$ can be replaced by $<_{\varphi_0}^x$,
this shows that $A$ is $\Delta^1_2$.
\end{proof}

To prove Theorem \ref{mainthmv2} we will use the technique developed
by A.W. Miller in \cite{miller89}. The idea is to replace $P$ in the
previous proof with a $\Pi^1_1$ relation $\hat P\subseteq
p_c(2^\omega)\times\cantor$ with the property that for all $f\in
p_c(\cantor)$ if
$$
(\exists x) \hat P(f,x)
$$
then $f$ ``codes'' some witness $x\in\cantor$ to this fact, more
precisely we will have
$$
(\exists x)\hat P(f,x)\iff (\exists x\in\Delta^1_1(f)) \hat P(f,x).
$$
Our maximal orthogonal set of measures will then be
$$
\hat A=\{\mu_f\in P_c(\cantor): (\exists x) \hat P(f,x)\}=\{\mu_f\in
P_c(\cantor): (\exists x\in\Delta^1_1(f)) \hat P(f,x)\},
$$
which will be $\Pi^1_1$ since $(\exists x\in\Delta^1_1(f))$ may be
replaced by a universal quantification, see e.g.
\cite[4.19]{manwei85}. What $f\in p_c(\cantor)$ specifically will
code is on the one hand the part of the inductive construction
witnessing that $f\in\hat A$, and on the other hand $x\in\cantor$
such that $M_x\simeq L_\delta$ for some limit $\delta<\omega_1$ in
which the inductive construction takes place. We need the following
facts for which B. Kasterman's thesis \cite{kastermans06} is an
excellent reference; see also \cite[\S 3]{kasstezha08}.

\begin{lemma}
(a) There are unboundedly many limit ordinals $\delta<\omega_1$ such
that there is $x\in L_{\delta+\omega}\cap\cantor$ such that
$M_x\simeq L_\delta$.

(b) If $M_x\simeq L_\delta$ for some limit $\delta<\omega_1$ then
there is $x'\in\Delta^1_1(x)$ such that $M_{x'}\simeq
L_{\delta+\omega}$.\label{kastlem}
\end{lemma}

{\it Remark.} (a) follows from \cite[Lemma 3.6]{kasstezha08}, (b)
from \cite[Lemma 3.5]{kasstezha08}.

\medskip

\noindent{\it Coding a real into a measure.} We now describe a way
of coding a given real $z\in\cantor$ into a measure $\mu\in
P_c(2^\omega)$. Given $\mu\in P_c(2^\omega)$ and $s\in 2^{<\omega}$
we let $t(s,\mu)$ be the lexicographically least $t\in 2^{<\omega}$
such that $s\subseteq t$, $\mu(N_{s^\smallfrown 0})>0$ and
$\mu(N_{s^\smallfrown 1})>0$ if it exists, and otherwise we let
$t(s,\mu)=\emptyset$.

Define inductively $t^\mu_n\in 2^{<\omega}$ by letting
$t^\mu_0=\emptyset$ and
$$
t^\mu_{n+1}=t({t^\mu_n}^\smallfrown 0,\mu).
$$
Note that since $\mu$ is non-atomic we have that
$\lh(t_{n+1}^\mu)>\lh(t_n^\mu)$; we let
$t_\infty^{\mu}=\bigcup_{n=0}^\infty t_n^\mu$. For $f\in
p_c(\cantor)$ and $n\in\omega\cup\{\infty\}$ we will write $t^f_n$
for $t^{\mu_f}_n$. Clearly the sequence $(t_n^f:n\in\omega)$ is
recursive in $f$. Define  $R\subseteq p_c(\cantor)\times \cantor$ as
follows:
\begin{align*}
R(f,z) \iff &(\forall n\in\omega) [z(n)=1 \leftrightarrow (f({t_n^f}^\smallfrown 0)=\frac{2}{3}f(t^f_n) \wedge f({t_n^f}^\smallfrown 1)=\frac{1}{3}f(t_n^f)) \\
&\wedge z(n)=0 \leftrightarrow f({t_n^f}^\smallfrown 0)=\frac{1}{3}
f(t_n^f)\wedge f({t_n^f}^\smallfrown 1)=\frac{2}{3} f(t_n^f)]
\end{align*}

\begin{codinglemma}
Given $z\in 2^\omega$ and $f\in p_c(2^\omega)$ there is $g\in
p_c(2^\omega)$ such that $\mu_f\approx\mu_g$ and $R(g,z)$. Moreover,
$g$ may be found in a recursive way  given $f$ and $z$: There is a
recursive function $G:p_c(2^\omega)\times 2^\omega\to p_c(2^\omega)$
such that $\mu_{G(f,z)}\approx \mu_f$ and $R(G(f,z),z)$ for all
$f\in p_c(2^\omega)$, $z\in 2^\omega$.\label{thecodinglemma}
\end{codinglemma}

\begin{proof}
We define $G(f,z)$ inductively. Suppose $G(f,z)\rest 2^{< n}$ has
been defined. Then for $s\in 2^n$ we let
$$
G(f,z)(s^\smallfrown i)= \left\{ \begin{array}{ll}
         \frac{2}{3}G(f,z)(s)  & \text{if $s=t_k^{f}$ for some $k\in\omega$, and } \\
                               &      z(k)=1, i=0;\\
         \frac{1}{3} G(f,z)(s) & \text{if $s=t_k^{f}$ for some $k\in\omega$, and }\\
                        &              z(k)=1, i=1;\\
         \frac{1}{3} G(f,z)(s) & \text{if $s=t_k^{f}$ for some $k\in\omega$, and }\\
                        &              z(k)=0, i=0;\\
         \frac{2}{3} G(f,z)(s) & \text{if $s=t_k^{f}$ for some $k\in\omega$, and }\\
                        &             z(k)=0, i=0;\\
         0              &   \text{if } f(s)=0;\\
         \\
         \dfrac{G(f,z)(s)}{f(s)} f(s^\smallfrown i) & \mbox{otherwise}.\end{array} \right.
$$
Define for $s\in 2^{<\omega}$
$$
\theta(s)=\frac{G(f,z)(s)}{f(s)}
$$
whenever $f(s)\neq 0$, and let $\theta(s)=0$ otherwise. Note that if
$x\neq t_\infty^f$ and $s\in 2^{<\omega}$ is the longest sequence
such that $s\subseteq x$ and $s\subseteq t_\infty^f$ then
$\theta(x\restrict n)$ is constant for $n>\lh(s)$. Let
$(s_i:i\in\omega)$ enumerate the set
$$
\{ s \in 2^{<\omega}: s\nsubseteq t_\infty^f\wedge s\restrict\lh(s)-1\subseteq t_\infty^f\}.
$$
Then for any Borel set $B$,
$$
\mu_{G(f,z)}(B)=\sum_{i=0}^\infty \theta(s_i) \mu_f(B\cap N_{s_i}),
$$
and since $\theta(s_i)=0$ if and only if $f(s_i)=0$ this shows that
$\mu_{G(f,z)}(B)=0$ if and only if $\mu_f(B)=0$. Thus
$\mu_g\approx\mu_f$, as required. In fact, if we define
$$
\hat\theta(x)=\lim_{n\to\infty} \theta(x\restrict n)
$$
if $x\neq t_\infty^f$ and $\hat\theta(t_\infty)=0$ then
$$
\frac{d\mu_g}{d\mu_f}=\hat\theta.
$$
Finally, it is clear from the definition of $G$ that $R(G(f,z),z)$,
and that $G$ is recursive.
\end{proof}

\begin{remark}
The relation $R(f,z)$ may be read as ``$f$ codes $z$''. The set
$$
\dom(R)=\{f\in p_c(\cantor): (\exists z) R(f,z)\}
$$
is $\Pi^0_1$ since deciding whether a given $f\in p_c(2^\omega)$
codes {\it some} $z\in\cantor$ only requires us to check for all $n$
that either $f({t^f_n}^\smallfrown i)=\frac{1}{3}f(t^f_n)$ or
$f({t^f_n}^\smallfrown i)=\frac{2}{3}f(t^f_n)$ holds, for $i=0$ and
$i=1$. Thus we may define a $\Pi^0_1$ function $r$ on $\dom(R)$ by
$$
r(\mu)=z\iff R(\mu,z).
$$
\end{remark}

\begin{proof}[Proof of Theorem \ref{mainthmv2}]
Work in $L$. We first define a maximal set of orthogonal measures by
induction on $\omega_1$, and then subsequently see that this set is
$\Pi^1_1$.

Let $\langle\mu_\alpha:\alpha<\omega_1\rangle$ be the sequence
defined in \ref{easyprop}. We will define a new sequence of measures
$\langle\nu_\alpha<\alpha<\omega_1\rangle$ such that
$\mu_\alpha\approx\nu_\alpha$, but where the resulting maximal
orthogonal set
$$
\hat A=\{\nu_\alpha:\alpha<\omega_1\}
$$
is $\Pi^1_1$. Suppose $\langle\nu_\alpha\in
P_c(\cantor):\alpha<\beta\rangle$ has been defined for some
$\beta<\omega_1$. Let $s_0\in p_c(\cantor)^{\leq\omega}$ be $<_L$
least such that
$$
\{\mu_\alpha:\alpha\leq\beta\}=\{\mu_{s_0(n)}:n\in\lh(s_0)\},
$$
and $s_0(0)$ is the $<_L$ largest element of
$\{s_0(n):n\in\lh(s_0)\}$. Let $x_0\in\cantor$ be $<_L$ least such
that $M_{x_0}\simeq L_\delta$ for some limit $\delta<\omega_1$ and
$s_0\in L_\delta$, and $x_0\in L_{\delta+\omega}$. That $x_0$ exists
follows from Lemma \ref{kastlem}. We let
$\nu_\beta=\mu_{G(s_0(0),\langle s_0,x_0\rangle)}$, where
$\langle\cdot,\cdot\rangle$ denotes some (fixed) reasonable
recursive way of coding a pair $(s,x)\in
p_c(\cantor)^{\leq\omega}\times \cantor$ as a single element of
$\cantor$. Note that $G(s_0(0),\langle s_0,x_0\rangle)\in
L_{\delta+\omega}$, since $G$ is recursive.

It is clear that
$$
\hat A=\{\nu_\alpha:\alpha<\omega_1\}
$$
is a maximal orthogonal set of measures. Thus it remains only to see
that $\hat A$ is $\Pi^1_1$.

We first define a relation $Q\subseteq
p_c(\cantor)^{\leq\omega}\times\cantor$, similar to $P$ in
Proposition \ref{easyprop}. We let $Q(s,x)$ if and only if
\begin{enumerate}[\hspace{1.6em}(a)]
\item $M_x$ is wellfounded and transitive, $M_x\models\sigma_0$, and
for some $m\in\omega$ we have $\pi_x(m)=s$.
\item $\{\mu_{s(n)}:n<\lh(s)\}$ is a set of orthogonal continuous measures.
\item For all $n<\lh(s)$ it holds that $s(n)$ is the $<_L$-smallest
code for a continuous measure orthogonal to all $s(k)$ for which
$s(k)<_L s(n)$.
\item $(\forall k>0) s(k)<_L s(0)$.
\end{enumerate}
That the relation $Q$ is $\Pi^1_1$ follows as in the proof of
Proposition \ref{easyprop}.

Now define a relation $\hat P\subseteq p_c(\cantor)\times\cantor$ by
letting $\hat P(f,x)$ if and only if
\begin{enumerate}
\item $M_x$ is wellfounded and transitive, $M_x\models\sigma_0$, and
for some $m\in\omega$ we have $\pi_x(m)=f$.
\item $f\in\dom(R)$, $r(f)=\langle s,w\rangle$ for some $(s,w)\in
p_c(\cantor)^{\leq\omega}\times\cantor$, and $Q(s,w)$.
\item $(\forall s'\in p_c(\cantor)^{\leq\omega})[\{s(n):n\in\lh(s)\}=\{s'(n):n\in\lh(s')\}\wedge s(0)=s'(0)]\longrightarrow \neg (s'<_L s)$
\item If $M_w\simeq L_{\delta'}$ then $w\in L_{\delta'+\omega}$,
and $w$ is $<_L$ least such that this holds and for some
$m\in\omega$, $\pi_w(m)=s$.
\item $f=G(s(0),\langle s,w\rangle)$.
\end{enumerate}
Conditions (1) and (2) are $\Pi^1_1$ and (5) is $\Pi^0_1$. If (1)
and (2) hold then (3) is equivalent to
\begin{align*}
(\forall s'\in p_c(\cantor)^{\leq\omega})((\forall l)(\exists l')
s(l)=s'(l')\wedge (\forall l)(\exists l') s'(l)=s(l')\wedge
s(0)=s'(0))\\
\longrightarrow ((\forall n_0)(\forall n_1) (\pi_x(n_0)=s\wedge
\pi_x(n_1)=s'\longrightarrow \neg n_1<^x_{\varphi_0} n_0)).
\end{align*}
which is a $\Pi^1_1$ predicate.

To verify that (4) is a $\Pi^1_1$ condition, define as in \cite[p.
170]{kanamori97} the restriction $M_x\restrict k$, for $x\in\cantor$
and $k\in\omega$, to be the $\mathcal L_{\epsilon}$ structure
$$
M_x\restrict k=(\{n: n\inx x k\},\epsilon_x).
$$
Assuming that (1)--(3) hold (4) is equivalent to the conjunction of
the following two conditions:
\begin{align*}
(\forall k)&((M_x\restrict k\models\sigma_0\wedge (\exists l) l\inx
x
k\wedge M_x\restrict l\simeq M_w)\longrightarrow\\
&((\exists n_0, n_1) n_0\inx x k\wedge n_1\inx x k\wedge
\pi_x(n_0)=s\wedge \pi_x(n_1)=w))
\end{align*}
and
\begin{align*}
(\forall w')(\forall k)&((M_{w'}\simeq M_x\restrict k\wedge
M_{w'}\models \sigma_0\wedge w'<_L w)\longrightarrow\\
&(\forall n_0, n_1)((n_0\inx x k\wedge n_1\inx x k)\longrightarrow
\pi_x(n_0)\neq s\wedge \pi_x(n_1)\neq w')),
\end{align*}
where $\simeq$ denotes isomorphism between $\mathcal L_\epsilon$
structures. Since $\simeq$, which is $\Sigma^1_1$, only occurs on
the left-hand side of the above implications, and since $<_L$ may be
replaced by using $<_{\varphi_0}^x$ as before, this shows that (4)
may be replaced by a $\Pi^1_1$ predicate, which proves that $\hat P$
is a $\Pi^1_1$ relation.

It is clear from the definition of $\hat P$ that
$$
\hat A=\{\mu_f\in P_c(\cantor): (\exists x)\hat P(f,x)\}.
$$
To see that $\hat A$ is $\Pi^1_1$, suppose that $\mu_f\in \hat A$.
Then $\hat P(f,x)$ for some $x$, and so $f\in\dom(R)$ and
$r(f)=\langle s,w\rangle$ where $M_w\simeq L_{\delta'}$. By Lemma
\ref{kastlem} there is $w'\in\Delta^1_1(w)$ such that $M_{w'}\simeq
L_{\delta'+\omega}$ and by condition (4) above $w\in
L_{\delta'+\omega}$. But then $\hat P(f,w')$ holds. Thus we have
shown that
$$
(\exists x) \hat P(f,x)\iff (\exists x\in\Delta^1_1(f))\hat P(f,x),
$$
which proves that $\hat A$ is $\Pi^1_1$.
\end{proof}

\section{Final remarks}

In this final section we consider two natural questions: What is the
cardinality of a maximal orthogonal family of measures? And is it
consistent that there is no co-analytic maximal orthogonal family of
measures?

\medskip

Both questions can be answered using the product measure
construction of Kechris and Sofronidis described in \S 2. Any
maximal orthogonal family in $P(\cantor)$ must have size $\mathfrak
c$ since there are $\mathfrak c$ many point measures. But even if we
only consider non-atomic measures we reach the same conclusion:

\begin{prop}
Let $A\subseteq P_c(2^\omega)$, $|A|<\mathfrak{c}$, be a set of
pairwise orthogonal measures. Then $A$ is not maximal orthogonal in
$P_c(\cantor)$. In fact, there is a product measure which is
orthogonal to all elements of $A$.
\end{prop}
\begin{proof} Suppose
$A=\{\nu_\alpha: \alpha<\mathfrak{\kappa}\}$, where
$\kappa<\mathfrak{c}$, is an orthogonal family. Since $E_I$ (as
defined in \S 2) has meagre classes, it follows by Mycielski's
Theorem (see e.g. \cite{gao08}) that there are perfectly many $E_I$
classes, and so we can find a sequence $(\mu_\alpha
:\alpha<\mathfrak{c})$ of orthogonal product measures. For each
$\nu_\alpha$ there can be at most countably many $\beta<\mathfrak c$
such that
$$
\mu_{\beta}\not\perp\nu_\alpha
$$
since $\ll$ is ccc below $\nu_\alpha$ (see e.g. the proof of
\cite[Theorem 3.1]{kecsof01}). Since $\kappa<\mathfrak c$ it follows
that there must be some $\beta<\mathfrak c$ such that $\mu_\beta$ is
orthogonal to all elements of $A$.
\end{proof}

\begin{prop}
If there is a Cohen real over $L$ then there is no $\Pi^1_1$ maximal
orthogonal set of measures.\label{cohen}
\end{prop}
\begin{proof}
We will use the following result of Judah and Shelah
\cite{judshe89}: If there is a Cohen real over $L$ then every
$\Delta^1_2$ set of reals is Baire measurable.

Suppose $A\subseteq P(\cantor)$ is a $\Pi^1_1$ maximal orthogonal
set of measures. Then define a relation $Q \subseteq \cantor\times
P(\cantor)^\omega$ by
$$
Q(x,(\nu_n))\iff (\forall n) (\nu_n\in
A\wedge\nu_n\not\perp\mu^x)\wedge
(\forall\mu)(\mu\not\perp\mu^x\longrightarrow (\exists n)
\nu_n\not\perp\mu)
$$
Then for each $x\in\cantor$ the section $Q_x$ is non-empty since $A$
is maximal. Using $\Pi^1_1$ uniformization, we obtain a function
$f:\cantor\to P(\cantor)^\omega$ having a $\Pi^1_1$ graph and such
that
$$
(\forall x) Q(x,f(x)).
$$
Now if $U\subseteq P(\cantor)^\omega$ is a basic open set then
\begin{align*}
x\in f^{-1}(U)&\iff (\exists (\nu_n)\in P(\cantor)^\omega)
f(x)=(\nu_n)\wedge (\nu_n)\in
U\\
&\iff (\forall (\nu_n)\in P(\cantor)^\omega) f(x)\neq(\nu_n)\vee
(\nu_n)\in U.
\end{align*}
Thus $f^{-1}(U)$ is $\Delta^1_2$, and so if there is a Cohen real
over $L$ then it has the property of Baire. It follows that $f$ is a
function with the Baire property. But then we may argue just as in
Kechris and Sofronidis' proof that no analytic set of measure is
maximal orthogonal and arrive at a contradiction: Indeed, $f$ is an
$E_I$-invariant assignment of countable subsets of $P(\cantor)$, and
$E_I$ is a turbulent equivalence relation, and so we must have that
$$
x\mapsto A(x)=\{f(x)(n):n\in\omega\}
$$
is constant on a comeagre set. This contradicts that
$$
(\forall x,x'\in\cantor) \neg (x E_I x')\implies \mu^x\perp\mu^{x'}
$$
and the ccc-below property of $\ll$.
\end{proof}

The natural relativization of Proposition \ref{cohen} gives us that
if for every $x\in\cantor$ there is a Cohen real over $L[x]$ then
there is no co-analytic (i.e., boldface $\mathbf\Pi^1_1$) maximal
set of orthogonal measures. We do not know what happens with the
complexity of maximal orthogonal sets if we add other types of
reals. In fact, we do not know the answer to the following:

\begin{question}
If there is a $\Pi^1_1$ maximal orthogonal set of measures, are all
reals constructible?
\end{question}

\bibliographystyle{amsplain}
\bibliography{pi11measures}

\providecommand{\bysame}{\leavevmode\hbox to3em{\hrulefill}\thinspace}
\providecommand{\MR}{\relax\ifhmode\unskip\space\fi MR }
% \MRhref is called by the amsart/book/proc definition of \MR.
\providecommand{\MRhref}[2]{%
  \href{http://www.ams.org/mathscinet-getitem?mr=#1}{#2}
}
\providecommand{\href}[2]{#2}
\begin{thebibliography}{10}

\bibitem{drake74}
F.R. Drake, \emph{Set theory, an introduction to large cardinals},
  North-Holland Publishing Company, 1974.

\bibitem{gao08}
S.~Gao, \emph{Invariant descriptive set theory}, Taylor and Francis, 2008.

\bibitem{gaozha05}
S.~Gao and J.~Zhang, \emph{Definable sets of generators in maximal cofinitary
  groups}, Advances in Mathematics \textbf{217} (2008), no.~2, 814--832.

\bibitem{judshe89}
H.~Judah and S.~Shelah, \emph{$\Delta^1_2$-sets of reals}, Annals of Pure and
  Applied Logic \textbf{42} (1989), no.~3, 207–223.

\bibitem{kakutani48}
S.~Kakutani, \emph{On equivalence of infinite product measures}, Annals of
  Mathematics \textbf{49} (1948), 214--226.

\bibitem{kanamori97}
A.~Kanamori, \emph{The higher infinite}, Springer Verlag, 1997.

\bibitem{kastermans06}
B.~Kastermans, \emph{Cofinitary groups and other almost disjoint families},
  Ph.D. thesis, University of Michigan, 2006,
  http://www.bartk.nl/Files/Papers/thesis.pdf.

\bibitem{kastermans09}
\bysame, \emph{The complexity of maximal cofinitary groups}, Proceedings of the
  American Mathematical Society \textbf{137} (2009), no.~1, 307--316.

\bibitem{kasstezha08}
B.~Kastermans, Juris Steprans, and Yi~Zhang, \emph{Analytic and coanalytic
  families of almost disjoint functions}, The Journal Symbolic Logic
  \textbf{73} (2008), no.~4, 1158--1172.

\bibitem{kechris73}
A.~Kechris, \emph{Measure and category in effective descriptive set theory},
  Annals of Pure and Applied Logic \textbf{5} (1973), 337--384.

\bibitem{kechris95}
\bysame, \emph{Classical descriptive set theory}, Graduate Texts in
  Mathematics, no. 156, Springer-Verlag, 1995.

\bibitem{kecsof01}
A.~Kechris and N.E. Sofronidis, \emph{A strong generic ergodicity property of
  unitary and self-adjoint operators}, Ergodic Theory and Dynamamical Systems
  \textbf{21} (2001), 1459--–1479.

\bibitem{manwei85}
R.~Mansfield and G.~Weitkamp, \emph{Recursive aspects of descriptive set
  theory}, Oxford University Press, 1985.

\bibitem{miller89}
A.W. Miller, \emph{Infinite combinatorics and definability}, Annals of Pure and
  Applied Logic \textbf{41} (1989), 179--203.

\bibitem{moschovakis80}
Y.~N. Moschovakis, \emph{Descriptive set theory}, North Holland, 1980.

\bibitem{preissrataj85}
D.~Preiss and J.~Rataj, \emph{Maximal sets of orthogonal measures are not
  analytic}, Proc. Amer. Math. Soc. \textbf{93} (1985), 471--476.

\end{thebibliography}

\end{document}